\documentclass[11pt]{article}

\usepackage{geometry}   
\usepackage{authblk}

\usepackage[utf8]{inputenc} 
\usepackage{hyperref}       
\usepackage{url}            
\usepackage{booktabs}       
\usepackage{amsfonts}       
\usepackage{nicefrac}       
\usepackage{microtype}      

\usepackage{amsmath,amsfonts,amssymb,amsthm,array}

\usepackage{algorithm}
\usepackage{caption}
\usepackage{subcaption}
\usepackage{algpseudocode}
\usepackage{mdframed} 
\usepackage{thmtools}

\usepackage{cases}
\usepackage{graphicx}
\usepackage{epstopdf}

\DeclareMathOperator*{\argmin}{arg\,min}

\newcommand{\R}{\mathbb{R}}

\def\<#1,#2>{\langle #1,#2\rangle}

\usepackage[colorinlistoftodos,bordercolor=orange,backgroundcolor=orange!20,linecolor=orange,textsize=scriptsize]{todonotes}

\declaretheorem[style=shaded,within=section]{definition}
\declaretheorem[style=shaded,sibling=definition]{theorem}

\declaretheorem[style=shaded,sibling=definition]{assumption}

\declaretheorem[style=shaded,sibling=definition]{lemma}

\title{\bf A matrix-free interior point continuous trajectory for linearly constrained convex programming}

\date{}

\author[1]{Xun Qian\thanks{Email: \texttt{xunqian2099@163.com}}}
\author[1]{Li-Zhi Liao\thanks{The work of L.-Z. Liao was supported in part by grants from Hong Kong Baptist University
		(FRG) and General Research Fund (GRF) of Hong Kong, email: \texttt{liliao@associate.hkbu.edu.hk}}}
\author[2]{Jie Sun\thanks{The work of J. Sun was partially supported
		by Australia Council Research under grant DP160102819 and the National Science Foundation of China under Grant 11671026, email:  \texttt{jie.sun@curtin.edu.au}}}

\affil[1]{Department of Mathematics, Hong Kong Baptist University, Hong Kong SAR, P. R. China}
\affil[2]{Faculty of Science and Engineering, Curtin University, Perth, Australia}

\begin{document}
	
	\maketitle
	
\begin{abstract}
Interior point methods for solving linearly constrained convex programming involve a variable projection matrix at each iteration to deal with
the linear constraints. This matrix often becomes ill-conditioned near the boundary of the feasible region that results in wrong search directions and
extra computational cost. A matrix-free interior point augmented Lagrangian continuous trajectory is therefore proposed and studied for linearly constrained
convex programming. A closely related ordinary differential equation (ODE) system is formulated. In this ODE system, the variable projection matrix
is no longer needed. By only assuming the existence of an optimal solution, we show that, starting from any interior feasible point,
(i) the interior point augmented Lagrangian continuous trajectory is convergent; and (ii) the limit point is indeed an optimal solution
of the original optimization problem. Moreover, with the addition of the strictly complementarity condition, we show that the associated Lagrange
multiplier converges to an optimal solution of the Lagrangian dual problem. Based on the studied ODE system, several possible search directions for discrete algorithms are proposed and discussed. 
\end{abstract}

\section{Introduction}

Consider the following linearly constrained convex programming problem
\begin{equation}\tag{\textrm{P}} \nonumber\label{primal}
	\begin{array}{rcl}
		\min & & f(x) \\
		{\rm s.t.\ }& &Ax= b, \ x_i\ge 0, \ i=1,...,s,
	\end{array}
\end{equation}
where $x\in \R^n$, $0\leq s \leq n$, $f(x)$ is convex and twice continuously differentiable, $A$ is an $m$ by $n$ matrix with full row rank.
As a blanket assumption, we assume that the optimal value for problem (\ref{primal}) is finite and attainable.
In addition, the following notations are used in this paper:
\begin{eqnarray*}
	&{\R^n_{s+}}=\{x\in \R^n| x_i\ge 0, \ 1\leq i\leq s\}, &\ {\R^n_{s++}}=\{x\in \R^n| x_i>0, \ 1\leq i\leq s \},
	\\
	&
	{\cal P^+}=\{x\in \R^n_{s+}|Ax=b\},\ \  \mbox{and}&\ {\cal P^{++}}=\{x\in \R^n_{s++}|Ax=b\}.
\end{eqnarray*}

\vskip 2mm

The Lagrangian function ${\cal L}:\ \R^n\times \R^m \times \R^s \to \R$ associated with (\ref{primal}) is defined for every $(x,y,z) \in \R^n\times \R^m \times \R^s$ by
$$
{\cal L}(x,y,z) = f(x) + y^T(Ax-b) - \sum\limits_{i=1}^s x_iz_i,
$$
and the Lagrangian dual problem associated with (\ref{primal}) is
\begin{equation}\tag{\textrm{D}} \nonumber\label{dual}
	\begin{array}{rcl}
		\max & & L(y,z) \\
		{\rm s.t.\ }& & z\ge 0,
	\end{array}
\end{equation}
where $L(y,z) = \inf\limits_{x} {\cal L}(x,y,z)$.

\vskip 2mm

In interior point methods, the linear constraint $Ax=b$ is maintained at each iteration, which means that an $m\times m$ linear system is generally involved.
For example, in the primal affine scaling algorithm \cite{Gonzaga02} for problem (\ref{primal}) with $s=n$, a projection matrix $P_{AX} = I - XA^T(AX^2A^T)^{-1}AX$
where $X = {\rm diag\ }(x)$ is used at each iteration. However, if $m$ is very large, then the inverse of the $m\times m$ matrix could be quite expensive. In addition,
as $x$ moves to the boundary of ${\cal P^+}$, the matrix $(AX^2A^T)$ could become ill-conditioned.
This paper focuses on avoiding this possible ill-conditioning problem caused by the linear constraints for problem (\ref{primal}).
Our strategy is to only maintain the positivity of $x$ while relaxing the equality constraint $Ax=b$.
This can be accomplished by combining the traditional interior point methods with the augmented Lagrangian method.
We call this new method the matrix-free interior point augmented Lagrangian method. Much attention will be paid to the trajectory of this method,
which is actually the solution of an ordinary differential equation (ODE) system.

\vskip 2mm

In particular, we are interested in the continuous solution trajectory of the following ODE system
\begin{equation}\label{odeipalm}
	\left\{
	\begin{array}{l}
		\frac{dx}{dt} = -U^{2}\left[\nabla f(x) + A^Ty + \sigma_1A^T(Ax-b)\right] , \ x(t_0) = x^0 \in {\R^n_{s++}}, \\
		\noalign{\vskip 1mm}
		\frac{dy}{dt} = \sigma_2(Ax-b), \ y(t_0) = y^0 \in {\R^m},  \\
	\end{array}\right.
\end{equation}
where
\begin{eqnarray*}
	&&\frac{1}{2}\leq \gamma <1,\ t_0\geq 0,\ \sigma_1>0,\ \sigma_2>0 , \\
	&&x\in \R^n_{s++},\ u\in \R^n,\ \{u_i\}_{i=1}^s=\{x_i^{\gamma}\}_{i=1}^s, \ u_i=1 \ \hbox{for} \ i=s+1,\dots,n, \\
	&&X={\rm diag\ }(x)\in \R^{n\times n},\ U={\rm diag\ }(u)\in \R^{n\times n}.
\end{eqnarray*}

\vskip 2mm

\noindent Let us explain where the ODE system (\ref{odeipalm}) comes from. Problem (\ref{primal}) can be written equivalently as
\begin{equation}\tag{\textrm{$P^{\prime}$}} \nonumber\label{primalpie}
	\begin{array}{rcl}
		\min & & f(x) + \delta_{\R^n_{s+}}(x) \\
		{\rm s.t.\ }& &Ax= b,\ x\in \R^n,
	\end{array}
\end{equation}
where $\delta_{\R^n_{s+}}(x)$ is the indicator function of $\R^n_{s+}$ defined as
\begin{equation}\nonumber
	\delta_{\R^n_{s+}}(x) =
	\begin{cases}
		0 &\mbox{if $x\in {\R^n_{s+}}$},\\
		+\infty &\mbox{otherwise}.
	\end{cases}
\end{equation}
Then the Lagrangian function $\tilde {\cal L}:\ \R^n\times \R^m \to \R$ associated with problem (\ref{primalpie}) is defined for every $(x,y) \in \R^n\times \R^m $ as
$$
\tilde {\cal L}(x,y) = f(x) + \delta_{\R^n_{s+}}(x) + y^T(Ax-b),
$$
and the augmented Lagrangian function ${\tilde {\cal L}}_{\sigma_1}:\ \R^n\times \R^m \to \R$ associated with problem (\ref{primalpie}) is defined for every $(x,y) \in \R^n\times \R^m $ as
\begin{equation}\label{auglagrangianf}
	{\tilde {\cal L}}_{\sigma_1}(x,y) = f(x) + \delta_{\R^n_{s+}}(x) + y^T(Ax-b) + \frac{\sigma_1}{2}\|Ax-b\|^2,
\end{equation}
where $\sigma_1 > 0$ is a parameter. To solve problem (\ref{primalpie}), the augmented Lagrangian method can be used. The augmented Lagrangian method was first proposed by Hestenes \cite{Hestenes69} and Powell \cite{Powell69}. Since then researchers have studied the augmented Lagrangian method in many different ways. For example, in \cite{Rockafellar76a,Rockafellar76b}, the augmented Lagrangian method and the proximal point method were studied by Rockafellar, and its convergence rate was obtained. In \cite{Tseng93}, the exponential method of multipliers, which operates like the usual augmented Lagrangian method except that it uses an exponential penalty function in place of the usual quadratic, was analyzed by Tseng and Bertsekas. In \cite{Iusem99}, the augmented Lagrangian methods and proximal point methods for convex optimization were considered by Iusem, and by using the generalized distances (Bregman distances and $\phi-$divergences), the generalized proximal point methods and the generalized augmented Lagrangian methods were proposed and studied. In \cite{Andreani07}, the augmented Lagrangian methods with general lower-level constraints were considered by Andreani $et$ $al.$, and the global convergence was obtained by using the constant positive linear dependence constraint qualification. In \cite{Sundef08}, the augmented Lagrangian method for nonlinear semidefinite programming was studied by Sun $et$ $al.$, and the linear convergence rate was obtained under the constraint nondegeneracy condition and the strong second-order sufficient condition. Zhao $et$ $al.$ \cite{Zhao10} considered a Newton-CG augmented Lagrangian method for solving semidefinite programming problems from the perspective of approximate semismooth Newton methods, and the convergence rate was analyzed by characterizing the Lipschitz continuity of the corresponding solution mapping at the origin. However, the method in \cite{Zhao10} may encounter numerical difficulty for degenerate semidefinite programming problems. In order to tackle this numerical difficulty, Yang $et$ $al.$ \cite{Yang15} employed a majorized semismooth Newton-CG augmented Lagrangian method coupled with a convergent 3-block alternating direction method of multipliers introduced by Sun $et$ $al.$ \cite{Sundef14}. In \cite{He15}, a new splitting version of the augmented Lagrangian method with full Jacobian decomposition for separable convex programming was proposed by He $et$ $al.$, and the worst-case convergence rate measured by the iteration complexity in both the ergodic and nonergodic senses was obtained.

\vskip 2mm

Applying the augmented Lagrangian method to problem (\ref{primalpie}), from any initial point $(x^0,y^0) \in {\R^n_{s+}}\times {\R^m}$, for $k = 1, 2, \cdots,$ we have the following iteration scheme

\vskip 2mm

\noindent Step 1. Compute
$$
x^{k+1} = \argmin_x {\tilde {\cal L}}_{\sigma_1}(x,y^k).
$$
Step 2. Compute
$$
y^{k+1} = y^k + \tau_k \sigma_1(Ax^{k+1} - b),
$$
where $\tau_k > 0$ is the step size.
In Step 1, we need to minimize a convex function over ${\R^n_{s+}}$. There are many methods to solve this subproblem. One of them is the interior point method. Particularly, we can use a first-order interior point method which is extended directly from the method in \cite{Tseng11}. In \cite{Tseng11}, Tseng $et$ $al.$ proposed a first-order interior point method for linearly constrained smooth optimization which unifies and extends
the first-order affine scaling method and the replicator dynamics method (see \cite{Bomze02}) for standard quadratic programming. Notice that the method in \cite{Tseng11} cannot be applied to the subproblem in Step 1 directly except $s=n$ in problem (\ref{primal}). However, we can just replace $x_i$ with $1$ in $X$ for $s+1\leq i \leq n$ to handle this, and the resulting search direction has the same form as $\frac{dx}{dt}$ in the ODE system (\ref{odeipalm}) (we restrict $\frac{1}{2}\leq \gamma <1$). In fact, if $\gamma=1$, the direction $\frac{dx}{dt}$ in the ODE system (\ref{odeipalm}) is the first-order affine scaling direction. The affine scaling algorithm was first introduced by Dikin \cite{Dikin67} in 1967. Since then many researchers have studied the affine scaling algorithm in many different ways. For instance, the affine scaling algorithm in linear programming was studied by Dikin \cite{Dikin74}, Saigal \cite{Saigal96}, Tseng and Luo \cite{Tseng92}, Tsuchiya \cite{Tsuchiya91}, and so on. The affine scaling continuous trajectory was also studied for linear programming, for example, by Adler and Monteiro \cite{Adler91}, Liao \cite{Liao14}, Megiddo and Shub \cite{Megiddo89}, Monteiro \cite{Monteiro91}. For convex quadratic programming and more general convex programming, the affine scaling algorithm was studied by Gonzaga and Carlos \cite{Gonzaga02}, Monteiro and Tsuchiya \cite{Monteiro98a}, Sun \cite{Sun93,Sun96}, Tseng $et$ $al.$ \cite{Tseng11}, Ye and Tse \cite{Ye89}. The affine scaling algorithm for convex semidefinite programming was studied by Qian $et$ $al.$ \cite{Qian18a}. Since this interior point method for the subproblem in Step 1 is an iterative method, in each iteration of the augmented Lagrangian method, $x$ can be updated several times and $y$ may be updated only once. Hence we may wonder at each iteration of the augmented Lagrangian method, whether $x$ could be only updated once. This is part of the motivation of this paper, and the ODE system (\ref{odeipalm}) is exactly the continuous realization of this idea.

\vskip 2mm

For simplicity, in what follows, $\|\cdot\|$ denotes the 2-norm. $C^k$ stands for the class of $k$th order continuously differentiable functions. Unless otherwise specified,
$x_j$ denotes the $j$th component of a vector $x$, and $I$ denotes the identity matrix, the dimension of $I$ is clear
from the context.
For any index subset $J \subseteq \{1, \ldots, n\}$,
we denote $x_{J}$ as the vector composed of those components of $x \in \R^n$ indexed by $j\in J$, and denote $Q_{JJ}$ as the submatrix of $Q$ composed by choosing the indexed rows and columns in $J$.

\vskip 2mm

The rest of this paper is organized as follows. First, a potential function for the ODE system (\ref{odeipalm}) is introduced in Section 2.
Furthermore, it is verified that the ODE system (\ref{odeipalm}) has a unique solution in $[t_0,+\infty )$.
With the help of this potential function, in Section 3, we prove that every accumulation point of the continuous trajectory $x(t)$ of the ODE system (\ref{odeipalm})
is an optimal solution for problem (\ref{primal}).
In Section 4, we show the strong convergence of the continuous trajectory $x(t)$ and verify that the limiting point has the maximal number of the positive components
in $\{x_1,\dots,x_s\}$ among the optimal solutions. Then with the addition of the strictly complementarity condition, we show that the associated Lagrange multiplier $(y(t), z(t)_S)$
converges to an optimal solution of the Lagrangian dual problem (\ref{dual}). In Section 5, two numerical examples are provided to show the performance of
the solution trajectory of the ODE system (\ref{odeipalm}). Several possible search directions for discrete algorithms which are derived from the ODE system (\ref{odeipalm}) are discussed briefly in Section 6.  Finally, some concluding remarks are drawn in Section 7.

\section{Properties of the solution trajectory of the ODE system (\ref{odeipalm})}
The following assumptions are made throughout  this paper.

\begin{assumption}\label{optb}
	There exists a point $x^* \in {\cal P^+}$ such that $f(x^*)$ is the optimal value of problem ($\ref{primal}$).
\end{assumption}

\begin{assumption}\label{fxinC2}
	$f(x) \in C^2$ on $\R^n_{s+}$.
\end{assumption}

\begin{assumption}\label{rankA}
	The matrix $A$ has full row rank $m$.
\end{assumption}

Theorem \ref{x(t)positive} below guarantees the existence and uniqueness for the solution of the ODE system (\ref{odeipalm}) .

\begin{theorem}\label{x(t)positive}
	For the ODE system (\ref{odeipalm}), there exists a unique solution $(x(t), y(t))$ with a maximal existence interval $[t_0, \alpha)$, in addition, $x_i(t)>0$ for $i=1,\dots,s$ on the existence interval.
\end{theorem}

\begin{proof}
	By Assumption \ref{fxinC2}, the right-hand side of the ODE system (\ref{odeipalm}) is locally Lipschitz continuous on $\R^n_{s++} \times {\R}^m$. From Theorem IV.1.2 in \cite{Bourbaki04},  there exists a unique solution $(x(t), y(t))$ for the ODE system (\ref{odeipalm}) on
	the maximal existence interval $[t_0, \alpha)$, for some $\alpha > t_0$ or $\alpha=+\infty$ such that $(x(t), y(t)) \in \R^n_{s++} \times {\R}^m$. Since $x(t) \in \R^n_{s++}$, $x_i(t)>0$ for $i=1,\dots,s$ on the existence interval.
	The proof is completed.
\end{proof}

Later in this section, it will be shown that $\alpha =+\infty $ (Theorem \ref{alpha+infty}).
To simplify the presentation, in the remaining of this paper, $x(t)$ (or $U(t)$) and $y(t)$ will be replaced by $x$ (or $U$) and $y$, respectively, whenever no confusion would occur.

\vskip 2mm

The next three lemmas lay the foundation for our potential function which will be introduced in (\ref{VxyDef}). Lemmas \ref{g(x)} and \ref{g(x)lambda} are Lemmas 10 and 11 in \cite{Qian18}, respectively.

\begin{lemma}\label{fconvex} (see \cite{ConvexopBoyd}) Suppose $f$ is differentiable (i.e., its gradient $\nabla f$ exists at each point in $dom f$). Then $f$ is convex if and only if $dom f$ is convex and
	\begin{equation}\label{f(x)convex}
		f(y) \geq f(x)+\nabla f(x)^T(y-x)
	\end{equation}
	holds for all $x$, $y\in dom f$.
\end{lemma}

\begin{proof}
	See Section 3.1.3 in \cite{ConvexopBoyd}.
\end{proof}

\begin{lemma}\label{g(x)}
	Let $a$ be any positive constant. Then for any scalar $x>0$, $g(x)=x-a-a \cdot \ln\frac{x}{a}\geq 0$ and $g(x)=0$ if and only if $x=a$.
	Furthermore, $g(x) \rightarrow +\infty$ as $x \rightarrow 0^+$ or $x \rightarrow +\infty$.
\end{lemma}

\begin{lemma}\label{g(x)lambda}
	Let $a$ be any positive constant and $1<r<2$. Then for any scalar $x>0$,
	$g(x)=\frac{1}{2-r}(x^{2-r}-a^{2-r})-\frac{a}{1-r}(\frac{1}{x^{r-1}}-\frac{1}{a^{r-1}})\geq 0$ and $g(x)=0$ if and only if $x=a$.
	Furthermore, $g(x) \rightarrow +\infty$ as $x \rightarrow 0^+$ or $x \rightarrow +\infty$.
\end{lemma}

Next we introduce a potential function for the ODE system (\ref{odeipalm}).
With the help of this potential function, the boundedness of the optimal solution set is no longer needed in the convergence proof of
the ODE system (\ref{odeipalm}).
Instead, only the weaker Assumption \ref{optb} is needed. In 1983, Losert and Akin \cite{Losert83} introduced a kind of potential function
for both the discrete and continuous dynamical systems in a classical model of population genetics. Their potential function
can be extended for our purpose. In \cite{Qian18}, a similar potential function is also used for the study of the generalized central paths
for problem (\ref{primal}). In order to define our potential function, we first introduce some notations.
For any $x\in \R_{s+}^n$, $B(x)=\{i\ | \ x_i > 0, \ i=1,\dots,s\}$ and $N(x)=\{i\ | \ x_i = 0, \ i=1,\dots,s\}$.
Obviously, for any $x\in \R_{s+}^n$, $B(x)\cap N(x)= \emptyset $ and $B(x)\cup N(x)= \{1, \dots, s\}$. Then the potential function $V(x, x^{\prime}, y, y^{\prime})$ for the ODE system
(\ref{odeipalm}) can be defined as
\begin{equation}\label{VxyDef}
	V(x, x^{\prime}, y, y^{\prime}) = I(x, x^{\prime}) + \frac{1}{2\sigma_2}\|y-y^{\prime}\|^2,
\end{equation}
where
$$
I(x, x^{\prime}) = \sum\limits_{i=s+1}^n \frac{1}{2}(x_i-x^{\prime}_i)^2 +   \hspace*{8cm}
$$
\begin{numcases}
	\textstyle \sum\limits_{i=1}^s(x_i-x^{\prime}_i)-\sum\limits_{i\in B(x^{\prime})}x^{\prime}_i \cdot \ln\frac{x_i}{x^{\prime}_i}& if $\gamma = \frac{1}{2},$ $B(x^{\prime}) \subseteq B(x),$ \label{potentialfun1a}\\
	\noalign{\vskip 1mm}
	\textstyle \sum\limits_{i=1}^s\frac{x_i^{2-2\gamma}-(x^{\prime}_i)^{2-2\gamma}}{2-2\gamma}&\nonumber \\
	\textstyle -\sum\limits_{i\in B(x^{\prime})}\frac{x^{\prime}_i}{1-2\gamma}\left(\frac{1}{x_i^{2\gamma-1}}-\frac{1}{(x^{\prime}_i)^{2\gamma-1}}\right)&
	if $\frac{1}{2} < \gamma<1,$ $B(x^{\prime}) \subseteq B(x),$\label{potentialfun1b}\\
	\noalign{\vskip 1mm}
	\textstyle +\infty &if $B(x^{\prime}) \nsubseteq B(x),$ \label{potentialfun1c}
\end{numcases}
and $(x, y) \in \R_{s+}^n \times {\R}^m$ is the variable, $(x^{\prime}, y^{\prime}) \in \R_{s+}^n \times {\R}^m$ is the parameter.

\vskip 2mm

A direct application of function $V(x, x^{\prime}, y, y^{\prime})$ in (\ref{VxyDef}) results in the following Theorems \ref{xytbound} and \ref{alpha+infty}.

\begin{theorem}\label{xytbound}
	Let $(x(t), y(t))$ be the solution of the ODE system (\ref{odeipalm}) on the maximal existence interval $[t_0, \alpha)$. Then $(x(t), y(t))$ is bounded on $[t_0, \alpha)$.
\end{theorem}

\begin{proof}
	We can choose the $x^*$ in Assumption \ref{optb}. Then there exist a $y^* \in {\R}^m$ and a $z^* \in \R^n$ such that $(x^*, y^*, z^*)$ satisfies the following KKT conditions for problem (\ref{primal})
	\begin{equation}\label{kktstar}
		\left\{
		\begin{array}{lll}
			Ax^*=b,\  x^*\in \R^n_{s+},\\
			(x^*)^Tz^*=0 , \ z^*\in \R^n_{s+},\\
			\nabla f(x^*) + A^Ty^* = z^*,\\
			z^*_i=0, \quad\hbox{for} \ s+1\leq i \leq n.\\
		\end{array}
		\right.
	\end{equation}
	Hence the convex function $f(x) + \delta_{\R^n_{s+}}(x) + (y^*)^T(Ax-b)$ attains the minimum at $x^*$ which implies
	\begin{equation}\label{Lxystar}
		f(x^*) \leq f(x) + (y^*)^T(Ax-b), \quad \forall x \in \R^n_{s+}.
	\end{equation}
	Since $x^* \in \R^n_{s+}$, we can define $V_1(x, y)$ as follows
	\begin{equation}\label{V1xyDef}
		V_1(x, y) = V(x, x^*, y, y^*), \quad \forall (x, y) \in \R^n_{s+} \times \R^m.
	\end{equation}
	From Theorem \ref{x(t)positive}, $V_1(x(t), y(t))$ is well defined on $[t_0, \alpha)$, and from Lemma \ref{fconvex} and (\ref{Lxystar}) we have
	\begin{eqnarray*}
		\frac{dV_1(x(t), y(t))}{dt} &=& (x-x^*)^TU^{-2}\cdot \frac{dx}{dt} + \frac{1}{\sigma_2}(y-y^*)^T\cdot \frac{dy}{dt} \\
		&=& (x^*-x)^T\left[\nabla f(x) + A^Ty + \sigma_1A^T(Ax-b) \right] + (y-y^*)^T (Ax-b) \\
		&=& (x^*-x)^T\nabla f(x) -\sigma_1\|Ax-b\|^2 - (y^*)^T(Ax-b) \\
		&\leq& f(x^*) - \left[f(x) + (y^*)^T(Ax-b)\right] - \sigma_1\|Ax-b\|^2 \\
		&\leq& - \sigma_1\|Ax-b\|^2 \leq 0,
	\end{eqnarray*}
	which indicates that
	\begin{equation}\label{V1xy}
		V_1(x(t), y(t)) \leq V_1(x^0, y^0), \ \ \forall t \in [t_0, \alpha).
	\end{equation}
	From Lemmas \ref {g(x)} and \ref{g(x)lambda}, we know that for $(x, y)\in \R^n_{s+} \times \R^m$ if $\|x\|\to +\infty$ or $\|y\| \to +\infty$,
	we must have $V_1(x, y)\rightarrow +\infty$. Hence from (\ref{V1xy}), $(x(t), y(t))$ must be bounded on $[t_0, \alpha)$.
\end{proof}

\begin{theorem}\label{alpha+infty}
	Let the maximal existence interval of the solution $(x(t), y(t))$ of the ODE system (\ref{odeipalm}) be $[t_0, \alpha)$. Then $\alpha =+\infty$.
\end{theorem}

\begin{proof}
	Assume $\alpha \neq +\infty$. From Theorem \ref{xytbound}, we know that there exists an $M>0$ such that $\|x(t)\|\leq M$ and $\|y(t)\|\leq M$ for all  $t \in [t_0, \alpha)$.
	Furthermore, from Assumption \ref{fxinC2}, we know that there exists an $L>0$ such that for every $i \in \{1, \ldots, s\}$, we have
	\begin{equation}\label{dxi1}
		\left|\frac{dx_i}{dt}\right| \leq Lx_i, \quad \forall t\in [t_0, \alpha),
	\end{equation}
	and for any $i \in \{s+1, \ldots, n\}$, $j \in \{1, \ldots, m\}$, we have
	\begin{equation}\label{dxi2}
		\left|\frac{dx_i}{dt}\right| \leq L \quad \mbox{and} \quad \left|\frac{dy_j}{dt}\right| \leq L \quad \forall t\in [t_0, \alpha).
	\end{equation}
	For any $i \in \{1, \ldots, n\}$, from inequalities (\ref{dxi1}), (\ref{dxi2}), and $\|x(t)\|\leq M$, we know that (without loss of generality we assume $M\geq 1$)
	\begin{equation}\label{dxiLM}
		\left|\frac{dx_i}{dt}\right| \leq LM \quad \forall t\in [t_0, \alpha),
	\end{equation}
	furthermore, $x(t)$ is continuous on $[t_0, \alpha)$,  and it is not hard to see that $\lim\limits_{t\rightarrow {\alpha}^-}x(t)$ exists. We denote this limit as $x(\alpha)$.
	Evidently $x(\alpha) \in \R^n_{s+}$. According to the Extension Theorem in $\S$2.5, \cite{Anosov88}, we know that the solution $(x(t), y(t))$ will go to the boundary of the open set $(0, +\infty)\times \R^n_{s++}\times \R^m$.
	But because of the hypothesis, $\alpha \neq +\infty$, and $(x(t), y(t))$ is bounded, so there must exist at least one $i \in \{1, \ldots, s\}$ such that $x_i(\alpha)=0$. From  inequality (\ref{dxi1}),
	we know that if $t\in [t_0, \alpha)$,
	$$
	\frac{dx_i}{x_i} \geq -Ldt.
	$$

	Integrating the above inequality, we have for every $t\in [t_0, \alpha)$
	$$
	\ln x_i(t)-\ln x_i(t_0)\geq -L(t-t_0).
	$$
	Since $x_i(t) \rightarrow x_i(\alpha)=0$ as $t\rightarrow \alpha ^-$, $\ln x_i(t)-\ln x_i(t_0)\rightarrow -\infty$ as $t\rightarrow \alpha ^-$, but $-L(t-t_0)\geq -L(\alpha-t_0)$.
	This is a contradiction.  Thus $\alpha=+\infty$ and the proof is completed.
\end{proof}

From Theorem \ref{alpha+infty}, we can define the limit set for the solution of the ODE system (\ref{odeipalm}). Let $(x(t), y(t))$ be the solution of the ODE system (\ref{odeipalm}), the limit set of $\{(x(t), y(t))\}$
can be defined as follows
\begin{eqnarray*}
	{\Omega}^1 (x^0, y^0)&=&\left\{ (p_1, p_2)\in {{\R}^n \times {\R}^m} \mid \exists \ \{t _k\}_{k=0}^{+\infty} \ {\rm with} {\displaystyle \lim_{k\to +\infty }t_k = +\infty} \right.\\
	&&\left. \mbox{ such that }
	{\displaystyle \lim_{k\to +\infty }}(x(t_k), y(t_k)) = (p_1, p_2) \right\}.
\end{eqnarray*}

\section{Optimality of the cluster point(s)}

From Theorem \ref{xytbound}, we know that the limit set ${\Omega}^1 (x^0, y^0)$ is nonempty. In this section, we will show that for any $(x^{(1)}, y^{(1)}) \in {\Omega}^1 (x^0, y^0)$, $x^{(1)}$ is an optimal solution for problem (\ref{primal}). First we need the following lemma.

\begin{lemma}(Barbalat's Lemma \cite{Slotine91})
	If the differentiable function $f(t)$ has a
	finite limit as $t\rightarrow+\infty $, and $\dot{f}$ is uniformly continuous, then $\dot{f}\rightarrow 0$  as $t\rightarrow+\infty$.
\end{lemma}

\begin{theorem}\label{optimality1}
	For any $(x^{(1)}, y^{(1)}) \in {\Omega}^1(x^0, y^0)$, $x^{(1)}$ is an optimal solution for problem (\ref{primal}).
\end{theorem}

\begin{proof}
	We can choose the $x^*$ in Assumption \ref{optb}, and $y^*$ as the corresponding Lagrange multiplier for constraint $Ax=b$ (see the proof of Theorem \ref{xytbound}). Note that when $i \in N(x^*)$, we have $x_i^*=0$. From Lemmas \ref{g(x)} and \ref{g(x)lambda}, it is easy to see that $V_1(x, y)\geq 0$ for all $(x, y)\in \R^n_{s++} \times {\R}^m$.
	Therefore, for all $t\in [t_0, +\infty)$, $V_1(x(t), y(t))$ is bounded below. This along with the fact that $\frac{dV_1(x(t), y(t))}{dt}\leq 0$ implies that $V_1(x(t), y(t))$ has a finite limit as $t\rightarrow+\infty $.
	
	\vskip 2mm
	
	From Theorem \ref{xytbound}, we know that the solution $(x(t), y(t))$ of the ODE system (\ref{odeipalm}) is contained in the bounded closed set
	$C=\{(x, y)\in \R^n \times \R^m | \|x\|\leq M, \ \|y\|\leq M \}$ for some $M>0$.
	From the proof of Theorem \ref{xytbound}, we have
	\begin{equation}\label{dv1t}
		\frac{dV_1(x(t), y(t))}{dt} = (x^*-x)^T\nabla f(x) -\sigma_1\|Ax-b\|^2 - (y^*)^T(Ax-b),
	\end{equation}
	which is continuously differentiable with respect to $x$ according to Assumption \ref{fxinC2}, hence when $C$ is compact, there must exist a constant $L_1>0$ such that
	$$
	\left|\frac{d V_1(x(t), y(t))}{dt}\mid_{t=t_1} - \frac{d V_1(x(t), y(t))}{dt}\mid_{t=t_2}\right| \leq L_1\|x(t_1) - x(t_2)\| = L_1 \| \int_{t_1}^{t_2} \frac{dx}{dt}dt \|.
	$$
	Using inequality (\ref{dxiLM}), we have
	$$
	\left|\frac{d V_1(x(t), y(t))}{dt}\mid_{t=t_1}-\frac{d V_1(x(t), y(t))}{dt}\mid_{t=t_2}\right| \leq \sqrt{n}L_1LM |t_1-t_2 |,
	$$
	thus $\frac{dV_1(x(t), y(t))}{dt}$ is uniformly continuous with respect to $t$. Hence from Barbalat's Lemma and (\ref{dv1t}), we have that
	\begin{equation}\label{limit0}
		\lim\limits_{t\rightarrow +\infty}(x^*-x)^T\nabla f(x) -\sigma_1\|Ax-b\|^2 - (y^*)^T(Ax-b)=0.
	\end{equation}
	For any $(x^{(1)}, y^{(1)}) \in {\Omega}^1(x^0, y^0)$, from the definition of ${\Omega}^1(x^0, y^0)$, we know that there exists a sequence $\{t_k\}_0^{+\infty}$
	with $t_k\rightarrow +\infty$ as $k\rightarrow +\infty$
	such that $x(t_k)\rightarrow x^{(1)}$ and $y(t_k) \to y^{(1)}$ as $k\rightarrow +\infty$. Then since $(x^*-x)^T\nabla f(x) -\sigma_1\|Ax-b\|^2 - (y^*)^T(Ax-b)$ is continuous at $x^{(1)}$,
	from (\ref{f(x)convex}), (\ref{Lxystar}), and (\ref{limit0}), we have
	\begin{eqnarray*}
		0 &=& \lim\limits_{t\rightarrow +\infty}(x^*-x)^T\nabla f(x) -\sigma_1\|Ax-b\|^2 - (y^*)^T(Ax-b) \\
		&=& \lim\limits_{k\rightarrow +\infty}(x^*-x(t_k))^T\nabla f(x(t_k)) -\sigma_1\|Ax(t_k)-b\|^2 - (y^*)^T(Ax(t_k)-b) \\
		&=& (x^*-x^{(1)})^T\nabla f(x^{(1)}) -\sigma_1\|Ax^{(1)}-b\|^2 - (y^*)^T(Ax^{(1)}-b) \\
		&\leq& f(x^*) - \left[ f(x^{(1)}) + (y^*)^T(Ax^{(1)}-b) \right] - \sigma_1\|Ax^{(1)} -b \|^2 \\
		&\leq& - \sigma_1\|Ax^{(1)} -b \|^2 \leq 0,
	\end{eqnarray*}
	which implies $Ax^{(1)} = b$ and
	$$
	f(x^*) = f(x^{(1)}) + (y^*)^T(Ax^{(1)}-b) = f(x^{(1)}).
	$$
	Furthermore, Theorem \ref{x(t)positive} and the definition of $\Omega^1(x^0, y^0)$, $x^{(1)}_i \geq 0$ for $i=1, ..., s$.
	Hence $x^{(1)}$ is an optimal solution for problem (\ref{primal}). Thus the theorem is proved.
\end{proof}

\section{Convergence of the solution trajectory of the ODE system (\ref{odeipalm})}
Now, it comes to the key results of the paper. To simplify the notation, we define
\begin{equation}\label{zxyDef}
	z(x, y) = \nabla f(x) + A^Ty + \sigma_1A^T(Ax-b),
\end{equation}
and $z(t) = z(x(t), y(t))$, where $(x(t), y(t))$ is the solution of the ODE system (\ref{odeipalm}). Let $S = \{1,\dots,s \}$ and ${\bar S} = \{ s+1,\dots,n \}$. Theorem \ref{strongcon1} below shows that $x(t)$ converges as $t \to +\infty $, and under the strict complementarity condition, $(y(t), z(t)_{S})$ converges to an optimal solution of the dual problem (\ref{dual}). First we need the weak convergence of the ODE system (\ref{odeipalm}) and the following lemma.

\begin{lemma} (\cite{Mangasaria88}){\label{gc}}
	Let $f: \R^n\rightarrow \R$ be a twice continuously differentiable  convex function. If $f(\cdot)$ is constant on a convex
	set $\bar{\Omega}\in \R^n$ then $\nabla f(\cdot)$ is constant on $\bar{\Omega} $.
\end{lemma}

\begin{theorem}\label{weakcon1}
	Let $(x(t), y(t))$ be the solution of the ODE system (\ref{odeipalm}). Then
	$$
	\lim\limits_{t\to +\infty} U^{2}\left[\nabla f(x) + A^Ty + \sigma_1A^T(Ax-b)\right] = 0,
	$$
	and
	$$
	\lim\limits_{t\to +\infty} Ax-b = 0.
	$$
\end{theorem}

\begin{proof}
	From (\ref{auglagrangianf}), we can define ${\tilde {\cal L}}_{\sigma_1}(x(t),y(t))$ and from (\ref{zxyDef}) we have
	\begin{equation}\label{dauglf}
		\frac{d{\tilde {\cal L}}_{\sigma_1}(x(t),y(t))}{dt} = -\|Uz(x, y)\|^2 + \sigma_2\|Ax-b\|^2.
	\end{equation}
	Furthermore, from Theorems \ref{xytbound} and \ref{optimality1}, it is not hard to see that
	\begin{equation}\label{limitauglf}
		\lim\limits_{t\to +\infty}{\tilde {\cal L}}_{\sigma_1}(x(t),y(t)) = f(x^*).
	\end{equation}
	
	From Theorem \ref{xytbound}, we know that the solution $(x(t), y(t))$ of the ODE system (\ref{odeipalm}) is contained in a bounded closed set
	$\{(x, y)\in \R^n \times \R^m | \|x\|\leq M, \ \|y\|\leq M \}$ for some $M>0$. Moreover, $-\|Uz(x, y)\|^2 + \sigma_2\|Ax-b\|^2$ is continuously differentiable with respect to $(x, y)$ according to Assumption \ref{fxinC2}, hence when $(x, y)$ is in this bounded closed set, there must exist a constant $L_2>0$ such that
	\begin{eqnarray*}
		&&\left|\frac{d {\tilde {\cal L}}_{\sigma_1}(x(t),y(t))}{dt}\mid_{t=t_1} - \frac{d {\tilde {\cal L}}_{\sigma_1}(x(t),y(t))}{dt}\mid_{t=t_2}\right| \leq L_2\|(x(t_1), y(t_1)) - (x(t_2), y(t_2))\| \\
		&&\leq L_2\Big(\|x(t_1) - x(t_2)\| + \|y(t_1)- y(t_2)\|\Big) =  L_2 \left[ \| \int_{t_1}^{t_2} \frac{dx}{dt}dt \| + \| \int_{t_1}^{t_2} \frac{dy}{dt}dt \| \right].
	\end{eqnarray*}
	Using (\ref{dxi2}) and (\ref{dxiLM}), we have
	$$
	\left|\frac{d {\tilde {\cal L}}_{\sigma_1}(x(t),y(t))}{dt}\mid_{t=t_1} - \frac{d {\tilde {\cal L}}_{\sigma_1}(x(t),y(t))}{dt}\mid_{t=t_2}\right| \leq L_2(\sqrt{n}LM + \sqrt{m}L)|t_1-t_2|,
	$$
	thus $\frac{d {\tilde {\cal L}}_{\sigma_1}(x(t),y(t))}{dt}$ is uniformly continuous with respect to $t$. Hence from (\ref{dauglf}), (\ref{limitauglf}), and Barbalat's Lemma, we have that
	$$
	\lim\limits_{t\to +\infty} -\|Uz(x, y)\|^2 + \sigma_2\|Ax-b\|^2 = 0.
	$$
	Furthermore, from Theorem \ref{optimality1}, we know
	$$
	\lim\limits_{t\to +\infty} Ax-b = 0,
	$$
	hence
	$$
	\lim\limits_{t\to +\infty} -\|Uz(x, y)\|^2 = 0,
	$$
	which implies
	$$
	\lim\limits_{t\to +\infty} U^{2}\left[\nabla f(x) + A^Ty + \sigma_1A^T(Ax-b)\right] = 0.
	$$
	Thus the proof is completed.
\end{proof}

\begin{theorem}\label{strongcon1}
	Let $(x(t), y(t))$ be the solution of the ODE system (\ref{odeipalm}). Then the following results hold.
	\begin{itemize}
		\item[(a)] For any $(x^{(1)}, y^{(1)}) \in {\Omega}^1(x^0, y^0)$, $(x^{(2)}, y^{(2)}) \in {\Omega}^1(x^0, y^0)$,
		and any optimal solution $(\tilde y, \tilde z)$ of problem (\ref{dual}), $\|y^{(1)} - \tilde y\| = \|y^{(2)} - \tilde y\|$.
		\item[(b)] $x(t)$ converges to an optimal solution of problem (\ref{primal}) which has the maximal number of positive components
		in $\{x_1,\dots,x_s\}$ among all optimal solutions.
		\item[(c)] If there exists a pair of primal and dual optimal solutions satisfying the strict complementarity, then $(y(t), z(t)_{S})$ converges to an optimal solution of problem (\ref{dual}).
	\end{itemize}
\end{theorem}

\begin{proof}
	(a) It should be noticed that from (\ref{kktstar}), $(y^*, z^*_{S})$ is an optimal solution of problem (\ref{dual}). Hence the optimal solution set of problem (\ref{dual}) is not empty.
	Assume $\|y^{(1)} - \tilde y\| < \|y^{(2)} - \tilde y\|$. We can define $V_2(x, y)$ as follows
	\begin{equation}\label{V2xyDef}
		V_2(x, y) = V(x, x^{(1)}, y, \tilde y),
	\end{equation}
	for any $(x, y) \in \R^n_{s+}\times \R^m$. From Theorems \ref{x(t)positive} and \ref{alpha+infty}, $V_2(x(t), y(t))$ is well defined on $[t_0, +\infty)$, and from Lemma \ref{fconvex} and Theorem \ref{optimality1} we have
	$$
	\frac{d V_2(x(t), y(t))}{dt} \leq f(x^{(1)}) - \left[ f(x) + (\tilde y)^T(Ax-b) \right] - \sigma_1\|Ax-b\|^2.
	$$
	Since $(\tilde y, \tilde z)$ is an optimal solution of problem (\ref{dual}), we have
	\begin{eqnarray*}
		f(x^{(1)}) = f(x^*) &=& L(\tilde y, \tilde z) = \inf\limits_{x} {\cal L}(x, \tilde y, \tilde z) \\
		&\leq& f(x) + (\tilde y)^T(Ax-b) - \sum\limits_{i=1}^s x_i{\tilde z}_i \\
		&\leq& f(x) + (\tilde y)^T(Ax-b),
	\end{eqnarray*}
	which implies
	$$
	\frac{d V_2(x(t), y(t))}{dt} \leq - \sigma_1\|Ax-b\|^2 \leq 0.
	$$
	Furthermore $V_2(x, y)$ is bounded below by zero, hence $\lim\limits_{t\to +\infty} V_2(x(t), y(t))$ exists.
	Since $(x^{(1)}, y^{(1)}) \in {\Omega}^1(x^0, y^0)$, there exists $\{ {\bar t}_k \}_{k=0}^{+\infty}$ with ${\bar t}_k \to +\infty$ such that
	$$
	{\displaystyle \lim _{k\to +\infty }}(x({\bar t}_k), y({\bar t}_k)) = (x^{(1)}, y^{(1)}).
	$$
	This along with (\ref{V2xyDef}) and (\ref{VxyDef}) implies
	\begin{equation}\label{V2limit}
		\lim\limits_{t\to +\infty} V_2(x(t), y(t)) = \lim\limits_{k \to +\infty} V_2(x({\bar t}_k), y({\bar t}_k)) = \frac{1}{2\sigma_2}\|y^{(1)} - \tilde y\|^2.
	\end{equation}
	Similarly, since $(x^{(2)}, y^{(2)}) \in {\Omega}^1(x^0, y^0)$ and $\|y^{(1)} - \tilde y\| < \|y^{(2)} - \tilde y\|$, we have
	$$
	\lim\limits_{t\to +\infty} V_2(x(t), y(t)) \geq \frac{1}{2\sigma_2}\|y^{(2)} - \tilde y\|^2 > \frac{1}{2\sigma_2}\|y^{(1)} - \tilde y\|^2,
	$$
	which contradicts with (\ref{V2limit}). Hence the hypothesis is not true and we have $\|y^{(1)} - \tilde y\| \geq \|y^{(2)} - \tilde y\|$. Similarly, we can show that $\|y^{(1)} - \tilde y\| \leq \|y^{(2)} - \tilde y\|$ and thus $\|y^{(1)} - \tilde y\| = \|y^{(2)} - \tilde y\|$.
	
	\vskip 2mm
	
	\noindent (b) Without loss of generality, we assume that the optimal solution $x^*$ in Assumption \ref{optb} has the maximal number of positive components in $\{x_1,\dots,x_s\}$ among all optimal solutions. For any $(x^{(1)}, y^{(1)}) \in {\Omega}^1(x^0, y^0)$, from Theorem \ref{optimality1}, we know that $x^{(1)}$ is an optimal solution of problem (\ref{primal}). From (\ref{V1xyDef}), (\ref{V1xy}), and Lemmas \ref{g(x)} and \ref{g(x)lambda}, it is easy to see that for each $i\in B(x^*)$ , $x_i(t)$ is bounded below by some positive constant $c_i$. So from Theorem \ref{optimality1}, $x^{(1)}$ must have the maximal number of positive components in $\{x_1,\dots,x_s\}$ among all optimal solutions and $B(x^{(1)}) = B(x^*)$, $N(x^{(1)}) = N(x^*)$.
	
	\vskip 2mm
	
	\noindent From (\ref{VxyDef}) and (\ref{V2xyDef}), we have
	\begin{equation}\nonumber
		V_2(x, y) = I(x, x^{(1)}) + \frac{1}{2\sigma_2}\|y - \tilde y\|^2,
	\end{equation}
	where
	\begin{eqnarray}\label{Ixx1}
		&&I(x, x^{(1)}) = \sum\limits_{i=s+1}^n \frac{1}{2}(x_i-{x^{(1)}}_i)^2 + \nonumber \\
		&&\left\{
		\begin{array}{ll}
			\sum\limits_{i\in N(x^{(1)})} x_i + \sum\limits_{i\in B({x^{(1)}})} ( x_i-{x^{(1)}}_i - {x^{(1)}}_i \cdot \ln\frac{x_i}{{x^{(1)}}_i}) & {\rm if} \ \gamma = \frac{1}{2}, \ B(x^{(1)}) \subseteq B(x),\\
			\noalign{\vskip 1mm}
			\sum\limits_{i\in N(x^{(1)})} \frac{x_i^{2-2\gamma}}{2-2\gamma} + \sum\limits_{i\in B({x^{(1)}})} {\bigg [ } \frac{x_i^{2-2\gamma}-({x^{(1)}}_i)^{2-2\gamma}}{2-2\gamma}& \\
			-\frac{{x^{(1)}}_i}{1-2\gamma}\left(\frac{1}{x_i^{2\gamma-1}}-\frac{1}{{x^{(1)}}_i^{2\gamma-1}}\right) {\bigg ]} & {\rm if} \ \frac{1}{2} < \gamma<1, \ B(x^{(1)}) \subseteq B(x),\\
			\noalign{\vskip 1mm}
			+\infty & {\rm if} \ B(x^{(1)}) \nsubseteq B(x).
		\end{array}\right. \nonumber  \\
	\end{eqnarray}
	From Lemmas \ref{g(x)}, \ref{g(x)lambda}, and (\ref{Ixx1}), it is straightforward to see that $I(x, x^{(1)}) \ge 0$ for any $x\in \{x \in \R^n_{s+}\ | \ x_i>0 \ {\rm if} \ i\in B(x^{(1)})\}$
	and $I(x, x^{(1)})=0 \iff x=x^{(1)}$. For any $(x^{(2)}, y^{(2)}) \in {\Omega}^1(x^0, y^0)$, we also have $B(x^{(2)}) = B(x^{(1)}) = B(x^*)$, hence $I(x, x^{(1)})$ is continuous at $x^{(2)}$. Then since $\lim\limits_{t\to +\infty} V_2(x(t), y(t))$ exists from the proof of (a), we have that
	$$
	\lim\limits_{t\to +\infty} V_2(x(t), y(t)) = \frac{1}{2\sigma_2}\|y^{(1)} - \tilde y\|^2 = I(x^{(2)}, x^{(1)}) + \frac{1}{2\sigma_2}\|y^{(2)} - \tilde y\|^2,
	$$
	this along with (a) implies that $I(x^{(2)}, x^{(1)}) = 0$. Thus $x^{(1)} = x^{(2)}$.
	
	\vskip 2mm
	
	\noindent (c) Without loss of generality, we assume that $x^*$ and $(y^*, z^*_S)$ satisfy the strict complementarity. According to Assumption \ref{fxinC2} and Lemma \ref{gc}, $\nabla f(x)$ is constant on the optimal solution set of problem (\ref{primal}). For $(x^{(1)}, y^{(1)}) \in {\Omega}^1(x^0, y^0)$, if $y^{(1)} = y^*$, then $y(t)$ converges to $y^*$ by (a). In this case, from $\nabla f(x^*) = \nabla f(x^{(1)})$, (\ref{kktstar}), and (b), we have that $z(t)_S$ converges to $z^*_S$. Obviously, $(y^*, z^*_S)$ is an optimal solution of problem (\ref{dual}) by (\ref{kktstar}).
	
	\vskip 2mm
	
	Now we consider the case of $y^{(1)} \neq y^*$. For $z(x^{(1)}, y^{(1)})$, from Theorem \ref{weakcon1} and $B(x^{(1)}) = B(x^*)$, we know that for $s+1\leq i\leq n$ or $i\in B(x^*)$, $z(x^{(1)}, y^{(1)})_i = 0$. By the strict complementarity, for $i \in N(x^*)$, $z^*_i > 0$, thus there exists a $0<\bar \lambda <1$ such that for $0\leq \lambda \leq \bar \lambda$ we have that $z^{[\lambda]}_i > 0$ for $i\in N(x^*)$ with
	$$
	z^{[\lambda]} = (1-\lambda)z^* + \lambda z(x^{(1)}, y^{(1)}).
	$$
	Moreover, for $s+1\leq i\leq n$ or $i\in B(x^*)$, $z^{[\lambda]}_i = 0$ for $0\leq \lambda \leq 1$. Let $y^{[\lambda]} = (1-\lambda) y^* + \lambda y^{(1)}$. Then from $\nabla f(x^*) = \nabla f(x^{(1)})$, we have $z(x^*, y^{[\lambda]}) = z^{[\lambda]}$. Hence for $0\leq \lambda \leq \bar \lambda$, $(x^*, y^{[\lambda]}, z^{[\lambda]})$ also satisfies the KKT conditions (\ref{kktstar}), which implies that $(y^{[\lambda]}, z^{[\lambda]}_S)$ is an optimal solution of problem (\ref{dual}) for $0\leq \lambda \leq \bar \lambda$. If there is a $(x^{(2)}, y^{(2)}) \in {\Omega}^1(x^0, y^0)$ such that $y^{(2)} \neq y^{(1)}$, then there exists a 2-dimensional plane passing through $y^{(1)}$, $y^{(2)}$, and $y^*$. On this 2-dimensional plane, we denote the circle centered at $y^{[\bar \lambda]}$ with radius $\|y^{(1)} - y^{[\bar \lambda]}\|$ by ${\cal C}_1$ and denote the circle centered at $y^*$ with radius $\|y^{(1)} - y^*\|$ by ${\cal C}_2$. Then from (a), we know that $y^{(2)}$ should be on the two circles ${\cal C}_1$ and ${\cal C}_2$ simultaneously. But this is impossible since the two circles ${\cal C}_1$ and ${\cal C}_2$ have only one point in common. Hence $y^{(1)} = y^{(2)}$ and $y(t)$ converges as $t\to +\infty$. This along with (b), we have that $z(t)$ converges to $z(x^{(1)}, y^{(1)})$.
	
	\vskip 2mm
	
	For $i\in N(x^*)$, if $z(x^{(1)}, y^{(1)})_i < 0$, then there exists an $N_1>t_0$ such that for $t\geq N_1$, $z(t)_i < 0$ which implies that $x^{(1)}_i \geq x(N_1) > 0$. But this contradicts with $N(x^{(1)}) = N(x^*)$. Thus for $i\in N(x^*)$, $z(x^{(1)}, y^{(1)})_i \geq 0$. Furthermore, for $s+1\leq i\leq n$ or $i\in B(x^*)$, $z(x^{(1)}, y^{(1)})_i = 0$. Then it is easy to see that $(x^{(1)}, y^{(1)}, z(x^{(1)}, y^{(1)}))$ satisfies the KKT conditions (\ref{kktstar}). Hence $(y^{{1}}, z(x^{(1)}, y^{(1)})_S)$ is an optimal solution of problem (\ref{dual}). Thus the proof is completed.
\end{proof}

\section{Numerical examples}

In this section, we show the performance of the solution trajectory of the ODE system (\ref{odeipalm}) on two examples. The first example is from \cite{Qian18}, and has the following form
\begin{equation} \label{counterexamplec2}
	\begin{array}{rcl}
		\min & & F(x,y) \\
		{\rm s.t.\ }& &y\geq0,
	\end{array}
\end{equation}
where $x\in \R$ and $y\in \R$ are variables, the optimal solutions are on the $x$-axis $\{(x,0) | x \in \R \}$, and $F(x,y) \in C^2$ . For the details of this example, see Section 4 in \cite{Qian18}. The problem (\ref{counterexamplec2}) is actually a specific instance in \cite{Gilbert05}, where a class of examples are proposed to show that the central path may fail to converge. The problem (\ref{counterexamplec2}) is a special case of problem (\ref{primal}) with $A = 0$.

\vskip 2mm

We plot the central path and the solution trajectories of the ODE system (\ref{odeipalm}) with different initial points for problem (\ref{counterexamplec2}). For our solution trajectories, we let $\gamma = 0.75$, $\sigma_1 = \sigma_2 = 1$, and $t_0=0$. A Matlab solver {\bf ode23s} is used to compute the trajectories of the ODE system (\ref{odeipalm}). For the central path, a Matlab code provided by Prof. Karas of \cite{Gilbert05} is used. In Fig. \ref{cpsp} (a), {\bf cp} represents the central path and {\bf sp} represents the solution trajectory of the ODE system (\ref{odeipalm}). Both paths have the same initial point $(-1.5, 1)$. Fig. \ref{cpsp} (b) is just a magnified display of the solution trajectory in Fig. \ref{cpsp} (a). Fig. \ref{sps} shows the solution trajectories with $6$ different initial points.

\begin{center}
	\begin{figure}[h]
		\center
		\includegraphics[scale = 0.4]{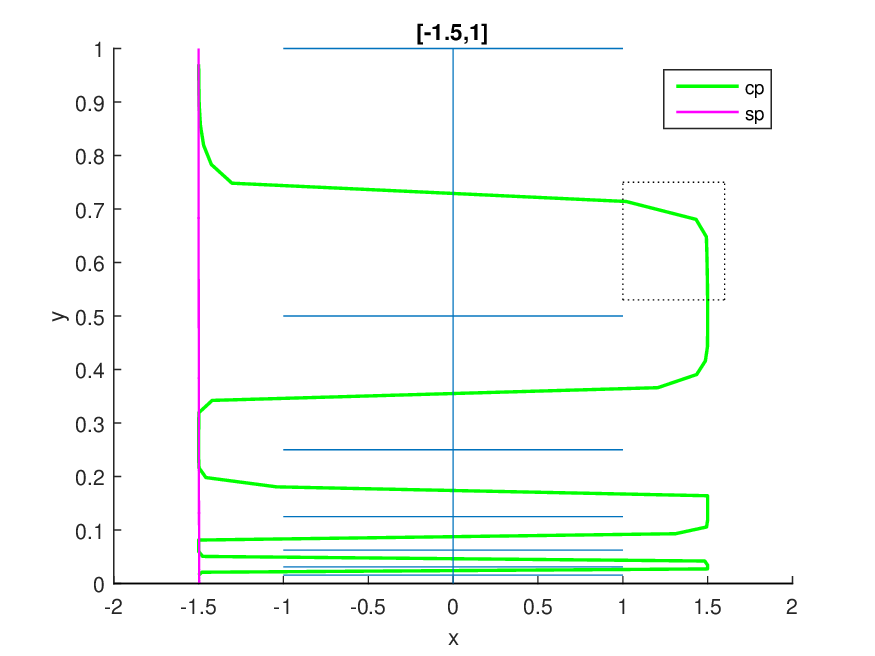}\includegraphics[scale = 0.4]{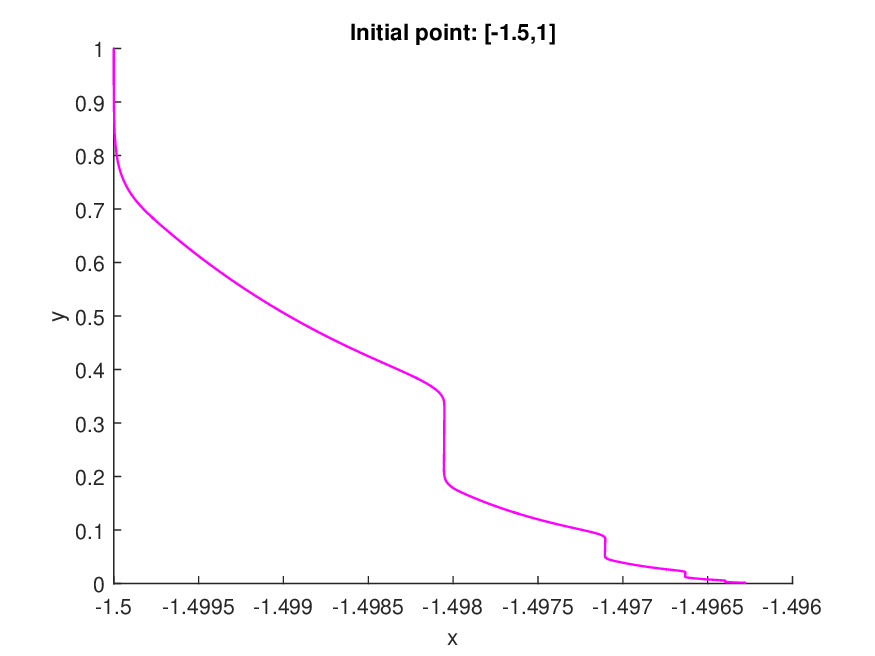}\\
		{\hskip -5mm {(a) Trajectories of cp and sp \hskip 25mm (b) Trajectory of sp}}
		\caption{Trajectories of the central path ({\bf cp}) and solution path ({\bf sp}) for problem (\ref{counterexamplec2})}\label{cpsp}
	\end{figure}
\end{center}

\begin{figure}[h]
	\center
	\includegraphics[scale = 0.4]{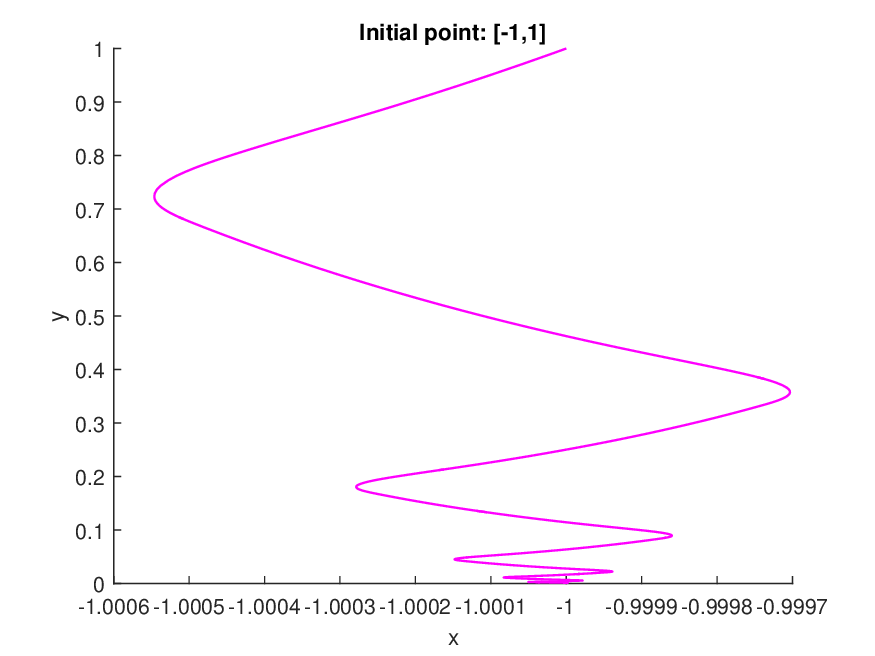}\includegraphics[scale = 0.4]{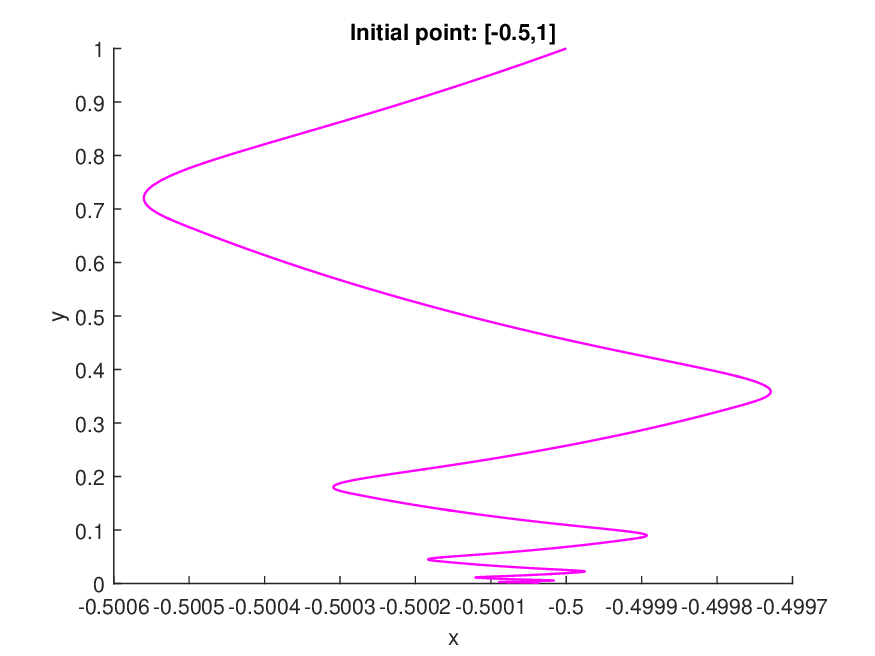}\\
	\includegraphics[scale = 0.4]{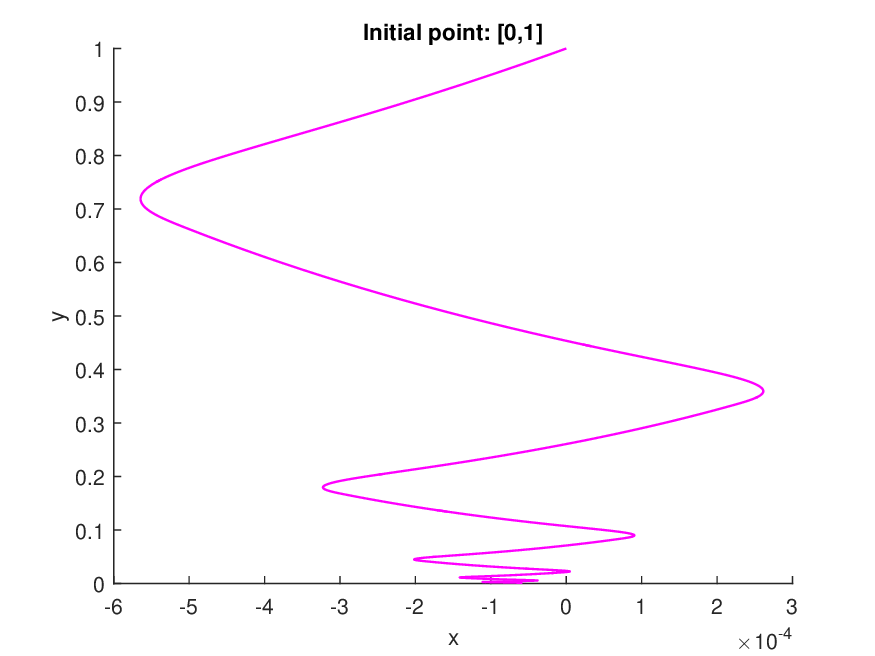}\includegraphics[scale = 0.4]{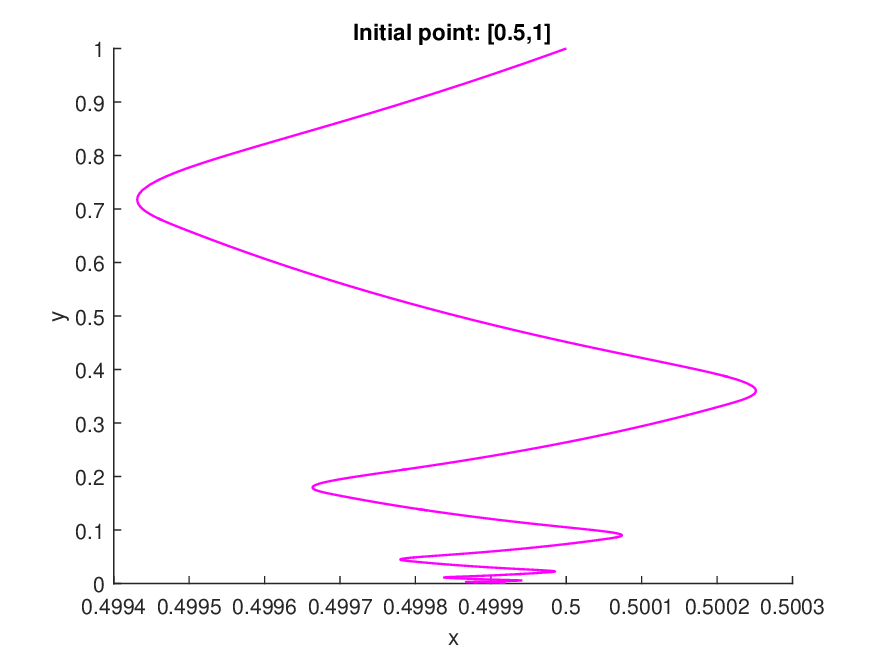}\\
	\includegraphics[scale = 0.4]{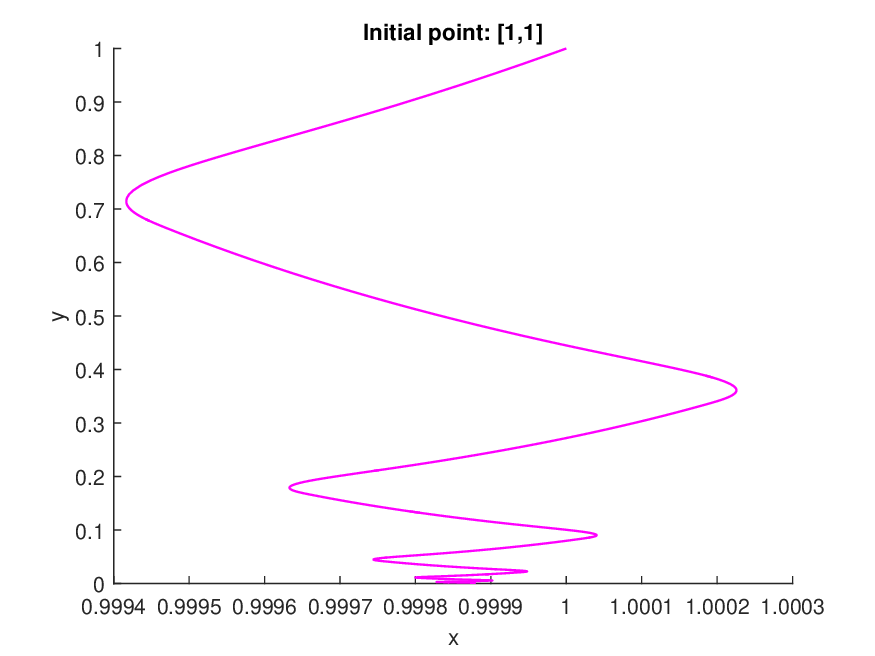}\includegraphics[scale = 0.4]{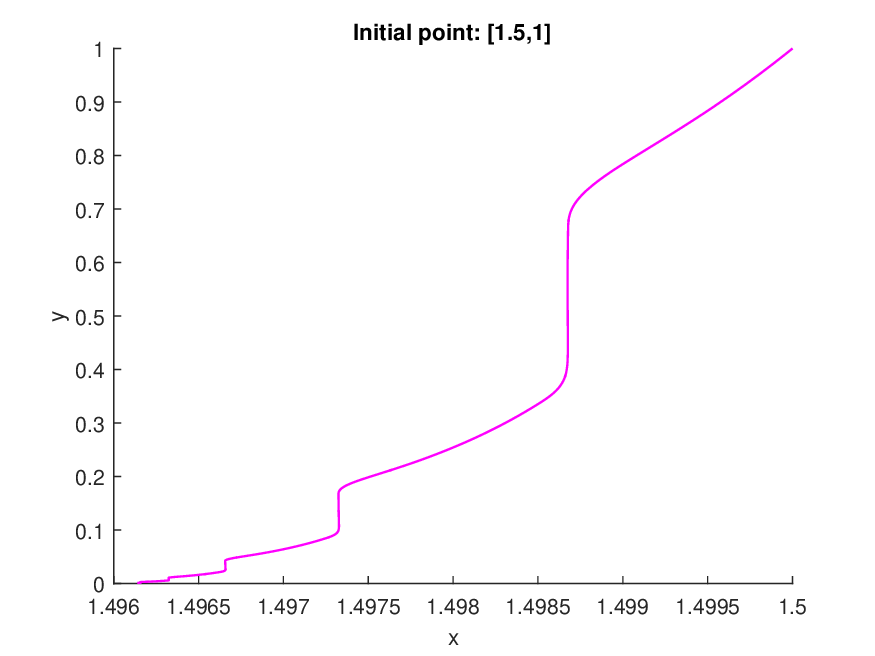}
	\caption{Trajectories of the solution paths with different initial points for problem (\ref{counterexamplec2})}\label{sps}
\end{figure}

Fig. \ref{cpsp} and Fig. \ref{sps} clearly show that for problem (\ref{counterexamplec2}), the central path is a zig-zag path with large swing, while the solution trajectories of the ODE system (\ref{odeipalm}) do converge.

\vskip 2mm

Next, we show the performance of the solution trajectory of the ODE system (\ref{odeipalm}) and primal (first-order) affine scaling trajectory for the following problem
\begin{equation}\label{example2}
	\begin{array}{rcl}
		\min & & f(x) = \frac{1}{24}\|x + c\|^4 + c^Tx \\
		{\rm s.t.\ }& &Ax= b, \ x\ge 0,
	\end{array}
\end{equation}
where $x\in \R^3$, $c = (1,1,1)^T$, $b = (1, 2)^T$, and
$$
A = \begin{pmatrix} 1 & 0 & 1 \\ 0 & 1 & 2\end{pmatrix}.
$$
It is easy to verify that $x^* = (0, 0, 1)^T$ is the unique optimal solution of problem (\ref{example2}), and $(x^*, y, z)$ satisfies the KKT conditions of problem (\ref{example2}) which are equivalent to
$$
\begin{pmatrix}y\\ z\end{pmatrix}
\in {\cal H} = \{ \begin{pmatrix}y\\z\end{pmatrix}
\in \R^5 | 3+y_1+2y_2=0,2+y_1=z_1, 2+y_2=z_2, z_3=0, z\geq 0  \}.
$$
This implies that ${\cal H}$ is actually the optimal solution set of the corresponding dual problem.

\vskip 2mm

The primal (first-order) affine scaling trajectory can be characterized by the following ODE system
\begin{equation}\label{odeaffine}
	\begin{array}{l}
		\frac{dx}{dt} = - XP_{AX}X\nabla f(x) , \ x(t_0) = {\tilde x}^0,
	\end{array}
\end{equation}
where $X = {\rm diag}(x)$, $P_{AX} = I - XA^T(AX^2A^T)^{-1}AX$, and ${\tilde x}^0$ needs to satisfy $A{\tilde x}^0 = b$ and ${\tilde x}^0> 0$. We plot the primal affine scaling trajectory and the solution trajectory of the ODE system (\ref{odeipalm}) from Matlab solver {\bf ode23s}. For the direction in (\ref{odeaffine}), the involved  linear system caused by the matrix $(AX^2A^T)^{-1}$ is solved by left division operator ``$\backslash$" in Matlab. For {\bf ode23s}, the relative error tolerance ``RelTol" is set as $1.0e-6$, and the absolute error tolerance ``AbsTol" is set as $1.0e-9$. For our solution trajectory, we set $\gamma = 0.75$, $\sigma_1 = \sigma_2 = 1$, and $t_0=0$. In Fig. \ref{affsp}, the magenta trajectories and the blue trajectories represent the coordinate behaviors of the solution trajectory $x(t)$ of the ODE system (\ref{odeipalm}) and the primal affine scaling trajectory, respectively. The two green dashed lines show the behaviors of our solution trajectory $y(t)$, and the green line represents $(y_1(t) + 2y_2(t))/3$, which should converge to $-1$ according to the definition of  ${\cal H}$. The initial point ${\tilde x}^0$ for the primal affine scaling trajectory is ${\tilde x}^0 = (0.5, 1, 0.5)^T$, and the initial point $(x^0, y^0)$ for the solution trajectory of the ODE system (\ref{odeipalm}) is $x^0 = (1, 1, 1)^T$ and $y^0 = (0, 1)^T$.

\vskip 2mm

\begin{figure}[h]
	\center
	\includegraphics[scale = 0.8]{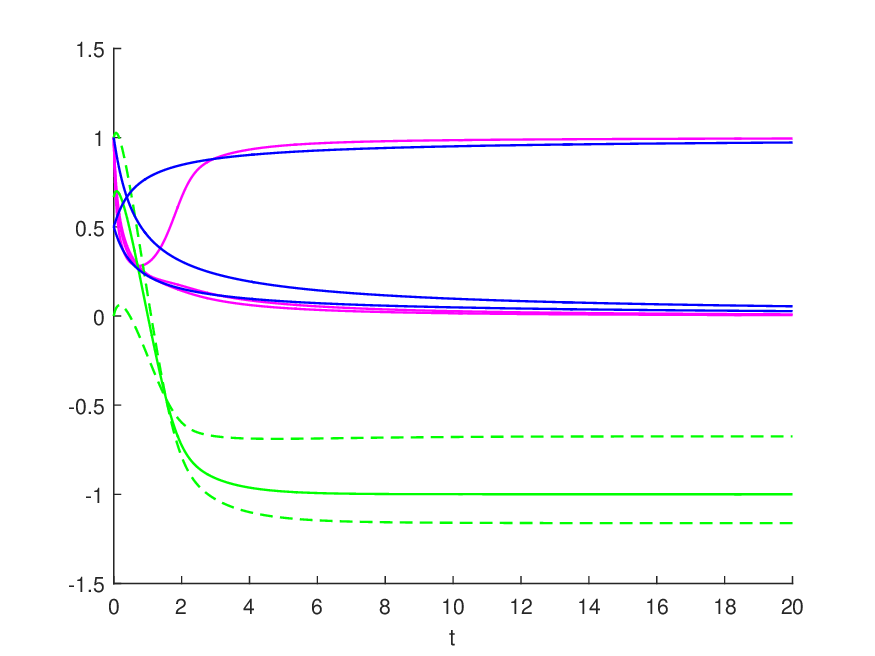}
	\caption{Trajectories of the primal affine scaling path ({\bf ap}) and solution path ({\bf sp}) for problem (\ref{example2})}\label{affsp}
\end{figure}

Fig. \ref{affsp} shows that both the primal affine scaling trajectory and our solution trajectory of the ODE system (\ref{odeipalm}) converge, and our solution trajectory actually converges faster than the primal affine scaling trajectory.

\vskip 2mm

Table \ref{Table1} lists the results at the end point $x(T)$ of the solution trajectory of the ODE system (\ref{odeipalm}) and the primal affine scaling trajectory for various integral intervals $[0, T]$. In Table \ref{Table1}, {\bf sp} represents the solution path of the ODE system (\ref{odeipalm}), and {\bf ap} represents the primal affine scaling path. $\min(x(T))$ means the minimum coordinate value of $x(T)$, and $\kappa(T)$ represents the condition number of the matrix $AX^2A^T$ (the inverse of $AX^2A^T$ is involved in the affine scaling direction) at point $x(T)$ for the primal affine scaling trajectory. In Table \ref{Table1}, NaN means that {\bf ode23s} failed due to the ill-conditioning.

\begin{table}[!t]
	\centering
	\caption{Performance of solution path ({\bf sp}) and affine scaling path ({\bf ap}) for problem (\ref{example2})}
	\begin{tabular}{|r||rr|rr|rr|r|}
		\noalign{\vskip 1mm}
		\hline
		\multicolumn{1}{|c|}{} &\multicolumn{2}{c|}{$\|x(T)-x^*\|_{\infty}$}&\multicolumn{2}{c|}{$\|Ax(T)-b\|$}&\multicolumn{2}{c|}{$\min(x(T))$}& $\kappa(T)$\\
		\hline
		$T$ & {\bf sp} & {\bf ap} & {\bf sp} & {\bf ap} & {\bf sp} & {\bf ap} & {\bf ap} \\
		\hline\hline
		10 & 2.7e-2&9.7e-2&1.6e-3&2.5e-12&1.5e-2&4.9e-2&1.2e+3\\
		$10^2$&5.2e-4&1.3e-2&3.5e-5&4.1e-11&2.2e-4&6.3e-3&7.9e+4\\
		$10^3$&5.6e-6&1.3e-3&4.4e-7&2.8e-10&2.3e-6&6.6e-4&7.2e+6\\
		$10^4$&5.6e-8&1.3e-4&5.6e-9&1.1e-8&2.3e-8&6.7e-5&7.0e+8\\
		$10^5$&5.5e-10&1.3e-5&1.3e-10&5.1e-7&2.2e-10&6.7e-6&7.0e+10\\
		$10^6$&5.3e-12&1.3e-6&7.8e-13&2.2e-6&2.2e-12&6.6e-7&7.0e+12\\
		$10^7$&4.9e-14&1.1e-6&4.0e-15&2.6e-6&2.0e-14&8.9e-8&5.9e+14\\
		$10^8$&4.0e-16&1.1e-6&5.0e-16&2.5e-6&1.6e-16&7.2e-9&3.5e+15\\
		$10^9$&4.0e-18&NaN&0&NaN&1.7e-18&NaN&NaN\\
		\hline
	\end{tabular}\label{Table1}
\end{table}

Table \ref{Table1} clearly shows that for the degenerate case where $AX^2A^T$ is singular at some point on the boundary of the feasible region, the interior point method could encounter ill-conditioning problem, while our method is matrix-free.

\vskip 2mm

\section{Further discussion}

In this section, we discuss about the possible search directions and algorithms which derived from the ODE system (\ref{odeipalm}). For the ODE system (\ref{odeipalm}), the simplest implementation way is the explicit Euler scheme. In fact, from any initial point $x^0 \in {\R^n_{s++}}$, this algorithm have the forms
\begin{equation}\label{iterationx}
	x^{k+1} = x^k + \alpha^k d_x^k, \ k=0,1,... ,
\end{equation}
and
\begin{equation}\label{iterationy}
	y^{k+1} = y^k + \alpha^k d_y^k, \ k=0,1,... ,
\end{equation}
where $\alpha^k$ is the step size,
$$
d_x^k = - U(x^k)^2 [\nabla f(x^k) + A^Ty^k + \sigma_1A^T(Ax^k-b)] = - U(x^k)^2z(x^k, y^k),
$$
and
$$
d_y^k = \sigma_2(Ax^k-b).
$$
In order to guarantee the optimality of the algorithm, some line search strategies and restrictions may be needed to decide the step size $\alpha^k$.

\vskip 2mm

At $(x^k, y^k)$, some semi-implicit directions can also be generated by the ODE system (\ref{odeipalm}). One possible way is that
\begin{eqnarray*}
	x^{k+1} - x^k &=& h_kd^k_x = -h_kU(x^k)^2[\nabla f(x^{k+1}) + A^Ty^k + \sigma_1A^T(Ax^k-b)] \\
	&\thickapprox& -h_kU(x^k)^2[\nabla f(x^k) + \nabla^2 f(x^k)(x^{k+1} - x^k) + A^Ty^k + \sigma_1A^T(Ax^k-b)],
\end{eqnarray*}
where $h_k$ represents the step size. This $d^k_x$ can be calculated approximately as
$$
d^k_x \thickapprox - U(x^k)\left[ I + h_kU(x^k)\nabla^2 f(x^k)U(x^k)  \right]^{-1}U(x^k)[\nabla f(x^k) + A^Ty^k + \sigma_1A^T(Ax^k-b)],
$$
which is a new search direction. If we want to follow the particular solution trajectory of the ODE system (\ref{odeipalm}) with the initial point $(x^k, y^k)$, $x^k$ should move a small $h_k$ along the search direction $d^k_x$. However, because of the residual, the point $x^k+h_kd^k_x$ may not be in ${\R^n_{s++}}$. Moreover, our purpose is not simulating any particular solution trajectory of the ODE system (\ref{odeipalm}), instead, we aim at finding the limit point of the solution path. Thanks to Theorem \ref{strongcon1}, the limit point of $x(t)$ from any initial point is an optimal solution for problem (\ref{primal}). Hence we can do a line search along this search direction, and $h_k$ can be regarded as a positive parameter. This direction $d^k_x$ involves an inverse of a $n\times n$ symmetric positive definite matrix, which needs too many calculations for large scale problems. Hence, in practice, the solution of the corresponding linear equations could be solved approximately.

\vskip 2mm

Another semi-implicit search direction $d^k_x$ can be generated from the following form
\begin{eqnarray*}
	x^{k+1} - x^k &=& h_kd^k_x = -h_kU(x^k)^2[\nabla f(x^{k+1}) + A^Ty^k + \sigma_1A^T(Ax^{k+1}-b)] \\
	&\thickapprox& -h_kU(x^k)^2[\nabla f(x^k) + \nabla^2 f(x^k)(x^{k+1} - x^k) + A^Ty^k \\
	&& + \sigma_1A^TA(x^{k+1} - x^k) + \sigma_1A^T(Ax^k-b)],
\end{eqnarray*}
where $h_k$ also represents the step size. Then $d^k_x$ can be calculated approximately as
\begin{eqnarray*}
	d^k_x &\thickapprox& - U(x^k)\left[ I + h_kU(x^k)(\nabla^2 f(x^k) + \sigma_1A^TA)U(x^k)  \right]^{-1} \\
	&& U(x^k)[\nabla f(x^k) + A^Ty^k + \sigma_1A^T(Ax^k-b)].
\end{eqnarray*}
Similar to the discussion of the first semi-implicit search direction, $h_k$ can be regarded as a positive parameter and the step size can be decided by a line search.

\vskip 2mm

Now in order to obtain more search directions, we make a division of the index set $\{1, 2, \cdots, n  \}$. Let $B_i$ $(i = 1, 2, \cdots, p)$ be $p$ nonempty index sets with $p\leq n$ such that $B_i\cap B_j = \varnothing$ for $i\neq j$ and $\bigcup\limits_{i=1}^p B_i = \{1, 2, \cdots, n  \}$. For the explicit Euler scheme of the ODE system (\ref{odeipalm}), the iteration (\ref{iterationx}) can be divided to $p$ sequential steps according to the index division, which means that $x^k_{B_i}$ moves to $x^k_{B_i} + \alpha^k(d^k_x)_{B_i}$ sequentially from $i=1$ to $p$ at each iteration. Then when we calculate $x^{k+1}_{B_i}$, we may use $x^{k+1}_{B_j}$ instead of $x^k_{B_j}$ for $1\leq j <i$, and this generates a new iteration. But there comes a question: how to decide the step size? One possible solving way is to calculate the step size $\alpha^k_i$ for each $x^k_{B_i}$ respectively. The $\alpha^k_i$ may be different from each other, which seems unreasonable since then the search direction at each iteration can not be derived from the ODE system (\ref{odeipalm}) directly, but this can be solved by bring in a weighted vector for the ODE system (\ref{odeipalm}). Actually if we replace $\frac{dx}{dt}$ in the ODE system (\ref{odeipalm}) by
$$
\frac{dx}{dt} = -WU^{2}\left[\nabla f(x) + A^Ty + \sigma_1A^T(Ax-b)\right] ,
$$
where $W = diag(w_1, w_2, \cdots, w_n)$ is a diagonal matrix with $w_j = \alpha^k_i$ if $j\in B_i$, then corresponding to this weighted ODE system, the new iteration can be derived from the following semi-implicit form
\begin{eqnarray*}
	x^{k+1}_{B_1} - x^k_{B_1} = h_k(d^k_x)_{B_1} &=& -h_k\left(WU(x^k)^{2}[\nabla f(x^k) + A^Ty^k + \sigma_1A^T(Ax^k-b)]\right)_{B_1} \\
	x^{k+1}_{B_2} - x^k_{B_2} = h_k(d^k_x)_{B_2} &=& -h_k\left(WU(x^k)^{2}[\nabla f(x^{k+1}_{B_1},x^{k}_{B_2},\cdots,x^k_{B_p}) + A^Ty^k \right.\\
	&& + \left. \sigma_1A^T(A(x^{k+1}_{B_1},x^k_{B_2},\cdots,x^k_{B_p})^T-b)]\right)_{B_2} \\
	&\vdots& \\
	x^{k+1}_{B_p} - x^k_{B_p} = h_k(d^k_x)_{B_p} &=& -h_k\left(WU(x^k)^{2}[\nabla f(x^{k+1}_{B_1},\cdots,x^{k+1}_{B_{p-1}},x^k_{B_p}) + A^Ty^k \right. \\
	&& + \left. \sigma_1A^T(A(x^{k+1}_{B_1},\cdots,x^{k+1}_{B_{p-1}},x^k_{B_p})^T-b)]\right)_{B_p}.
\end{eqnarray*}
It is evident that with step size equals to one, the above form generates the new iteration. Furthermore, it is easy to see that by similar proofs, the same results of the ODE system (\ref{odeipalm}) hold for this weighted ODE system.

\vskip 2mm

For the division $B_i$ $(i=1, 2, \cdots, p)$ of index set $\{1, 2, \cdots, n  \}$, there are many other search directions which can be generated from the ODE system (\ref{odeipalm}). Here we enumerate two more possible implementations. One possible implementation way is the following semi-implicit form
\begin{eqnarray*}
	x^{k+1}_{B_1} - x^k_{B_1} = h_k(d^k_x)_{B_1} &=& -h_k\left(U(x^k)^{2}[\nabla f(x^{k+1}_{B_1},x^{k}_{B_2},\cdots,x^k_{B_p}) + A^Ty^k \right.\\
	&& + \left. \sigma_1A^T(A(x^{k+1}_{B_1},x^k_{B_2},\cdots,x^k_{B_p})^T-b)]\right)_{B_1} \\
	x^{k+1}_{B_2} - x^k_{B_2} = h_k(d^k_x)_{B_2} &=& -h_k\left(U(x^k)^{2}[\nabla f(x^{k}_{B_1},x^{k+1}_{B_2},\cdots,x^k_{B_p}) + A^Ty^k \right.\\
	&& + \left. \sigma_1A^T(A(x^{k}_{B_1},x^{k+1}_{B_2},\cdots,x^k_{B_p})^T-b)]\right)_{B_2} \\
	&\vdots& \\
	x^{k+1}_{B_p} - x^k_{B_p} = h_k(d^k_x)_{B_p} &=& -h_k\left(U(x^k)^{2}[\nabla f(x^{k}_{B_1},\cdots,x^{k}_{B_{p-1}},x^{k+1}_{B_p}) + A^Ty^k \right. \\
	&& + \left. \sigma_1A^T(A(x^{k}_{B_1},\cdots,x^{k}_{B_{p-1}},x^{k+1}_{B_p})^T-b)]\right)_{B_p}.
\end{eqnarray*}
Similarly, the $d^k_x$ can be calculated approximately as
\begin{eqnarray}\label{secdkx}
	(d^k_x)_{B_i} &\thickapprox& - U(x^k)_{B_{i}B_{i}}\left[ I + h_kU(x^k)_{B_{i}B_{i}}(\nabla^2 f(x^k) + \sigma_1A^TA)_{B_{i}B_{i}}U(x^k)_{B_{i}B_{i}}  \right]^{-1} \nonumber \\
	&& U(x^k)_{B_{i}B_{i}}[\nabla f(x^k) + A^Ty^k + \sigma_1A^T(Ax^k-b)]_{B_i},
\end{eqnarray}
for $i = 1, 2, \cdots, p$, and $h_k$ can be regarded as a positive parameter and the step size can be decided by a line search.  It is interesting to see that for this implementation, the iteration for $x^k$ can also be divided to $p$ sequential steps. Hence similar to the discussion before, when we calculate $x^{k+1}_{B_i}$, we can use $x^{k+1}_{B_j}$ instead of $x^{k}_{B_j}$ for $1\leq j <i$, and decide the step size $\alpha^k_i$ by the line search respectively. However, this new search direction in this new iteration scheme seems not be able to be derived from the ODE system (\ref{odeipalm}), since each $\alpha^k_i$ may be not equal to $h_k$. Hence in this situation, we may regard this new iteration scheme as a partially updated algorithm in the sense that at each iteration, for the $d^k_x$ in (\ref{secdkx}), we only move one $x^k_{B_i}$ along $(d^k_x)_{B_i}$, and proceed this process from $i=1$ to $p$ sequentially and then loop.

\vskip 2mm

In the above discussion, we only considered the $d^k_x$. For $d^k_y$, we can use $\sigma_2(Ax^k - b)$ or $\sigma_2(Ax^{k+1} - b)$. Different choice can lead to different algorithms. It should be noticed that, some of the above search directions involve the inverse of the $n\times n$ matrix or the block sub-matrix, but the condition number of them can be controlled by the restriction of the hyperparameters. In this section, we only discuss the possible search directions for discrete algorithms briefly, and the study of these algorithms can be the future research.

\vskip 2mm

\section{Concluding remarks}
Ill-conditioning subproblems arise in the computation of many interior point algorithms for linearly constrained convex programming problems.
In this paper, we have developed a matrix-free interior point augmented Lagrangian continuous trajectory
for these problems. This continuous trajectory can be viewed as the solution of the ODE system (\ref{odeipalm}) and only matrix-vector product is required
in this ODE system (\ref{odeipalm}).
It has been proved that starting from any interior feasible point, this continuous trajectory would always converge to some optimal
solution of the original problem. Therefore, a stepping stone is laid for developing a matrix-free numerical scheme to follow this convergent continuous trajectory. In particular, we discuss several possible search directions for discrete algorithms briefly, and further study could be the future research.

\bibliographystyle{plain}
\bibliography{IPALM_ref.bib}

\end{document}